\documentclass[letterpaper]{article}

\usepackage{charter}
\usepackage[T1]{fontenc}
\usepackage[utf8]{inputenc}
\usepackage[dvipsnames]{xcolor}
\usepackage[bookmarksnumbered]{hyperref}
\usepackage{enumitem}
\usepackage[super]{nth}
\usepackage[gen]{eurosym}
\usepackage{graphicx}
\usepackage{amssymb}
\usepackage{amsmath}
\usepackage{stmaryrd}
\usepackage{algorithm}
\usepackage{algorithmic}
\usepackage{dsfont}
\usepackage{mathtools}
\usepackage{tikz}
\usepackage[numbers]{natbib}
\usepackage{booktabs}
\usepackage{amsfonts}
\usepackage{amsthm}
\usepackage[font=small]{caption}
\usepackage{titlesec}
\usepackage{fancyhdr}
\usepackage{lipsum}
\usepackage[framemethod=TikZ]{mdframed}
\usepackage{pifont}
\usepackage{pdflscape}
\usepackage{dsfont}
\usepackage{parskip}

\renewcommand{\leq}{\leqslant}
\renewcommand{\geq}{\geqslant}
\renewcommand{\epsilon}{\varepsilon}

\DeclareMathOperator*{\argmin}{\mathrm{arg\,min}}
\DeclareMathOperator*{\argmax}{\mathrm{arg\,max}}



\hypersetup{
    colorlinks=true,
    linkcolor={red!60!black},
    citecolor={green!40!black}
}

\newtheorem{theorem}{Theorem}[section]

\newtheorem{remark}[theorem]{Remark}

\begin{document}

\begin{center}
 {\Large Complexity of Linear Minimization\\ and Projection on Some Sets}
\end{center}

\vspace{7mm}

\noindent\textbf{Cyrille W.\ Combettes}\textsuperscript{ 1 3}\hfill\href{mailto:cyrille@gatech.edu}{\ttfamily cyrille@gatech.edu}\\
\\
\textbf{Sebastian Pokutta}\textsuperscript{ 2 3}\hfill\href{mailto:pokutta@zib.de}{\ttfamily pokutta@zib.de}\\
\\
{\small\textsuperscript{1 }\emph{School of Industrial and Systems Engineering, Georgia Institute of Technology, USA}}\\
{\small\textsuperscript{2 }\emph{Institute of Mathematics, Technische Universit\"at Berlin, Germany}}\\
{\small\textsuperscript{3 }\emph{Department for AI in Society, Science, and Technology, Zuse Institute Berlin, Germany}}

\vspace{7mm}

\begin{center}
\begin{minipage}{0.85\textwidth}
\begin{center}
 \textbf{Abstract}
\end{center}
\vspace{3mm}
  {\small The Frank-Wolfe algorithm is a method for constrained optimization that relies on linear minimizations, as opposed to projections. Therefore, a motivation put forward in a large body of work on the Frank-Wolfe algorithm is the computational advantage of solving linear minimizations instead of projections. However, the discussions supporting this advantage are often too succinct or incomplete. In this paper, we review the complexity bounds for both tasks on several sets commonly used in optimization. Projection methods onto the $\ell_p$-ball, $p\in\left]1,2\right[\cup\left]2,+\infty\right[$, and the Birkhoff polytope are also proposed.}
\end{minipage}
\end{center}

\vspace{2mm}

\section{Introduction}

We consider the constrained optimization problem
\begin{align}
 \min_{x\in\mathcal{C}}f(x),\label{pb}
\end{align}
where $\mathcal{C}\subset\mathbb{R}^n$ is a compact convex set and $f\colon\mathbb{R}^n\rightarrow\mathbb{R}$ is a smooth function. Among all general purpose methods addressing problem~\eqref{pb}, the Frank-Wolfe algorithm \cite{fw56}, a.k.a.\ conditional gradient algorithm \cite{levitin66}, has the particularity of never requiring projections onto $\mathcal{C}$. It uses linear minimizations over $\mathcal{C}$ instead and is therefore often referred to as a \emph{projection-free} algorithm in the literature, in the sense that it does not call for solutions to quadratic optimization subproblems.

Thus, a motivation put forward in a large body of work on the Frank-Wolfe algorithm is the computational advantage of solving linear minimizations instead of projections. However, only a few works actually provide examples. On the other hand, the complexities of linear minimizations over several sets are available in \cite{hazan12,jaggi13fw,garber16phd}, but they do not always (accurately) discuss the complexities of the respective projections. Therefore, while it is intuitive that a linear minimization is simpler to solve than a projection in general, a complete quantitative assessment is necessary to properly motivate the projection-free property of the Frank-Wolfe algorithm.

\paragraph{Contributions.} We review the complexity bounds of linear minimizations and projections on several sets commonly used in optimization: the standard simplex, the $\ell_p$-balls for $p\in\left[1,+\infty\right]$, the nuclear norm-ball, the flow polytope, the Birkhoff polytope, and the permutahedron. These sets are selected because linear minimizations or projections can be solved very efficiently, rather than by resorting to a general purpose method, in which case the analysis is less interesting. We also propose two methods for projecting onto the $\ell_p$-ball and the Birkhoff polytope respectively, and we analyze their complexity. Computational experiments for the $\ell_1$-ball and the nuclear norm-ball are presented.

\begin{remark}
 We would like to stress that, while it is possible that a projection-based algorithm requires less iterations than the Frank-Wolfe algorithm to find a solution to problem~\eqref{pb}, our goal here is to demonstrate its advantage in terms of per-iteration complexity. We discuss the Frank-Wolfe algorithm and some successful applications in Section~\ref{sec:fw}.
\end{remark}

\section{Preliminaries}

\subsection{Notation and definitions}

We work in the Euclidean space $\mathbb{R}^n$ or $\mathbb{R}^{m\times n}$ equipped with the standard scalar product $\langle x,y\rangle=x^\top y$ or $\langle X,Y\rangle=\operatorname{tr}(X^\top Y)$. We denote by $\|\cdot\|$ the norm induced by the scalar product, i.e., the $\ell_2$-norm $\|\cdot\|_2$ or the Frobenius norm $\|\cdot\|_{\operatorname{F}}$ respectively. 
For any closed convex set $\mathcal{C}\subset\mathbb{R}^n$, the projection operator onto $\mathcal{C}$, the distance function to $\mathcal{C}$, and the diameter of $\mathcal{C}$, all with respect to $\|\cdot\|$, are denoted by $\operatorname{proj}(\cdot,\mathcal{C})$, $\operatorname{dist}(\cdot,\mathcal{C})$, and $\operatorname{diam}(\mathcal{C})$ respectively.

For every $i,j\in\mathbb{N}$ such that $i\leq j$, the brackets $\llbracket i,j\rrbracket$ denote the set of integers between (and including) $i$ and $j$. For all $x\in\mathbb{R}^n$ and $i,j\in\llbracket1,n\rrbracket$ such that $i\leq j$, $[x]_i$ denotes the $i$-th entry of $x$ and $[x]_{i:j}=([x]_i,\ldots,[x]_j)^\top\in\mathbb{R}^{j-i+1}$.
The signum function is $\operatorname{sign}\colon\lambda\in\mathbb{R}\mapsto1$ if $\lambda>0$, $-1$ if $\lambda<0$, and $0$ if $\lambda=0$.
The characteristic function of an event $E$ is $\mathds{1}_E=1$ if $E$ is true, else $0$. 
The indicator function of a set $\mathcal{C}\subset\mathbb{R}^n$ is $\iota_\mathcal{C}\colon x\in\mathbb{R}^n\mapsto0$ if $x\in\mathcal{C}$, else $+\infty$.
Operations on vectors in $\mathbb{R}^n$, such as $\operatorname{sign}(x),|x|,x^p,\max\{x,y\},xy$, that are conventionally applied to scalars, are carried out entrywise and return a vector in $\mathbb{R}^n$.
The shape of $0$ and $1$ will be clear from context, i.e., a scalar or a vector.
The identity matrix in $\mathbb{R}^{n\times n}$ is denoted by $I_n$. The matrix with all ones in $\mathbb{R}^{m\times n}$ is denoted by $J_{m,n}$, and by $J_n$ if $m=n$.

We adopt the real-number infinite-precision model of computation. The complexity of a computational task is the number of arithmetic operations necessary to execute it. We ran the experiments on a laptop under Linux Ubuntu 20.04 with Intel Core i7-10750H. The code is available at \href{https://github.com/cyrillewcombettes/complexity}{\textcolor{blue}{\ttfamily https://github.com/cyrillewcombettes/complexity}}.

\subsection{The Frank-Wolfe algorithm}
\label{sec:fw}

The Frank-Wolfe algorithm (FW) \cite{fw56}, a.k.a.\ conditional gradient algorithm \cite{levitin66}, is a first-order projection-free algorithm for solving constrained optimization problems~\eqref{pb}. It is presented in Algorithm~\ref{fw}.

\begin{algorithm}[h]
\caption{Frank-Wolfe (FW)}
\label{fw}
\begin{algorithmic}[1]
\REQUIRE Start point $x_0\in\mathcal{C}$, step-size strategy $(\gamma_t)_{t\in\mathbb{N}}\subset\left[0,1\right]$.
\FOR{$t=0$ \textbf{to} $T-1$}
\STATE$v_t\leftarrow\argmin\limits_{v\in\mathcal{C}}\,\langle v,\nabla f(x_t)\rangle$\label{fw:lmo}
\STATE$x_{t+1}\leftarrow x_t+\gamma_t(v_t-x_t)$\label{fw:new}
\ENDFOR
\end{algorithmic}
\end{algorithm}

At each iteration, FW minimizes the linear approximation of $f$ at $x_t$ over $\mathcal{C}$ (Line~\ref{fw:lmo}), i.e.,
\begin{align*}
 \min_{v\in\mathcal{C}}f(x_t)+\langle v-x_t,\nabla f(x_t)\rangle,
\end{align*}
and then moves in the direction of a solution $v_t\in\mathcal{C}$ with a step-size $\gamma_t\in\left[0,1\right]$ (Line~\ref{fw:new}). This ensures that the new iterate $x_{t+1}=(1-\gamma_t)x_t+\gamma_tv_t\in\mathcal{C}$ is feasible by convexity, and there is no need for a projection back onto $\mathcal{C}$. For this \emph{projection-free} property, FW has encountered numerous applications, including solving traffic assignment problems \cite{leblanc75}, performing video co-localization \cite{joulin15video}, or, e.g., developing adversarial attacks \cite{chen20adversarial}.

When $f$ is convex, FW converges at a rate $f(x_t)-\min_\mathcal{C}f=\mathcal{O}(1/t)$ for different step-size strategies \cite{fw56,dunn78,jaggi13fw}, which is optimal in general \cite{canon68,jaggi13fw}. Faster rates can be established under additional assumptions on the properties of $f$ or the geometry of $\mathcal{C}$ \cite{levitin66,guelat86,garber15,kerdreux21}. Recently, several variants have also been developed to improve its performance \cite{lacoste15,lan16cgs,garber16dicg,freund17,braun19lazy,combettes20boost}. When $f$ is nonconvex, \cite{lacoste16} showed that FW converges to a stationary point at a rate $\mathcal{O}(1/\sqrt{t})$ in the gap $\max_{v\in\mathcal{C}}\langle x_t-v,\nabla f(x_t)\rangle$ \cite{hearn82}, which has inspired a line of work in stochastic optimization \cite{reddi16,yurtsever19,xie20,combettes20ada,zhang20}.

Lastly, note that FW is also popular for the natural sparsity of its iterates with respect to the vertices of $\mathcal{C}$, as $x_t\in\operatorname{conv}\{x_0,v_0,\ldots,v_{t-1}\}$ \cite{clarkson10,hazan08,lacoste13svm,luise19,combettes19cara}.

\section{Projections versus linear minimizations}
\label{sec:res}

The Frank-Wolfe algorithm avoids projections by computing linear minimizations instead. In Table~\ref{tab:lmo}, we summarize the complexities of a linear minimization and a (Euclidean) projection on several sets commonly used in optimization. That is, we compare the complexities of solving
\begin{align}
 \min_{x\in\mathcal{C}}\,\langle x,y\rangle
 \quad\text{and}\quad
 \min_{x\in\mathcal{C}}\|x-y\|.\label{lmo-proj}
\end{align}
When an exact solution cannot be computed directly, we compare to the complexity of finding an $\epsilon$-approximate solution; note that the two objectives in~\eqref{lmo-proj} are homogeneous. For the projection problem, it means to solve $\min_{x\in\mathcal{C}}\|x-y\|^2$ using any method but the Frank-Wolfe algorithm, since it would go against the purpose of this paper. Note however that the Frank-Wolfe algorithm can generate a solution with complexity $\mathcal{O}(\operatorname{iter}(\mathcal{C})\operatorname{diam}(\mathcal{C})^2/\epsilon^2)$, where $\operatorname{iter}(\mathcal{C})$ denotes the complexity of an iteration, which amounts to that of a linear minimization over $\mathcal{C}$. When addressing problem~\eqref{pb}, solving projection subproblems via the Frank-Wolfe algorithm is known as conditional gradient sliding \cite{lan16cgs}.

\begin{table}[h]
\vspace{1mm}
 \caption{Complexities of linear minimizations and (Euclidean) projections on some sets commonly used in optimization. We denote by $x^*$ a solution, $\rho=p\sup_{t\in\mathbb{N}}\|x_t\|_{2(p-1)}^{p-1}\|x_t\|_2<+\infty$ where $(x_t)_{t\in\mathbb{N}}$ is the sequence generated by Algorithm~\ref{hgz}, by $\nu$ and $\sigma_1$ the number of nonzero entries and the top singular value of $-Y$ respectively, and by $\epsilon>0$ the additive error in the objective of~\eqref{lmo-proj} when an approximate solution is computed. The constant $d_z$ is defined in~\eqref{dz} and $\tilde{\mathcal{O}}$ hides polylogarithmic factors in $m$ and $n$.}
\label{tab:lmo}
\centering{
 \begin{tabular}{llll} 
  \toprule
  \textbf{Set $\mathcal{C}$}
  &\textbf{Linear minimization}&\textbf{Projection}&\textbf{Reference}\\
  \midrule
  $\ell_p$-ball, $p\in\{1,2,+\infty\}$
  &$\mathcal{O}(n)$&$\mathcal{O}(n)$&Sections~\ref{sec:l1}--\ref{sec:l2inf}\\
  $\ell_p$-ball, $p\in\left]1,2\right[\cup\left]2,+\infty\right[$
  &$\mathcal{O}(n)$&$\mathcal{O}(n\rho^2\|y-x^*\|_2^2/\epsilon^2)$&Section~\ref{sec:lp}\\
  Nuclear norm-ball&$\mathcal{O}(\nu\ln(m+n)\sqrt{\sigma_1}/\sqrt{\epsilon})$&$\mathcal{O}(mn\min\{m,n\})$&Section~\ref{sec:nuc}\\
  Flow polytope&$\mathcal{O}(m+n)$&$\tilde{\mathcal{O}}(m^3n+n^2)$&Section~\ref{sec:flow}\\
  Birkhoff polytope
  &$\mathcal{O}(n^3)$&$\mathcal{O}(n^2d_z^2/\epsilon^2)$&Section~\ref{sec:birk}\\
  Permutahedron&$\mathcal{O}(n\ln(n))$&$\mathcal{O}(n\ln(n)+n)$&Section~\ref{sec:perm}\\
  \bottomrule
 \end{tabular}
 }
\end{table}

We now provide details for the complexities reported in Table~\ref{tab:lmo}. Slightly abusing notation though we may have $\operatorname{card}(\argmin_{x\in\mathcal{C}}\langle x,y\rangle)>1$, we write $\argmin_{x\in\mathcal{C}}\langle x,y\rangle=x^*$ instead of $\argmin_{x\in\mathcal{C}}\langle x,y\rangle\ni x^*$.

\subsection{The \texorpdfstring{$\ell_1$}{l1}-ball and the standard simplex}
\label{sec:l1}

Let $\{e_1,\ldots,e_n\}$ denote the standard basis in $\mathbb{R}^n$. The $\ell_1$-ball is 
\begin{align*}
 \{x\in\mathbb{R}^n\mid\|x\|_1\leq1\}
 =\left\{x\in\mathbb{R}^n\;\middle\vert\;\sum_{i=1}^n|[x]_i|\leq1\right\}
 =\operatorname{conv}\{\pm e_1,\ldots,\pm e_n\},
\end{align*}
and the standard simplex is 
\begin{align*}
 \Delta_n=\{x\in\mathbb{R}^n\mid\langle x,1\rangle=1,x\geq0\}=\operatorname{conv}\{e_1,\ldots,e_n\}.
\end{align*}
A projection onto the $\ell_1$-ball amounts to computing a projection onto the standard simplex, for which the most efficient algorithms have a complexity $\mathcal{O}(n)$; see \cite{condat16} for a review. On the other hand, linear minimizations are available in closed form: for all $y\in\mathbb{R}^n$,
\begin{align*}
 \argmin_{x\in\Delta_n}\,\langle x,y\rangle
 =e_{i_{\min}}
 \quad\text{and}\quad
 \argmin_{\|x\|_1\leq1}\,\langle x,y\rangle
 =-\operatorname{sign}([y]_{i_{\max}})e_{i_{\max}},
\end{align*}
where $i_{\min}\in\argmin_{i\in\llbracket1,n\rrbracket}[y]_i$ and $i_{\max}\in\argmax_{i\in\llbracket1,n\rrbracket}|[y]_i|$.
Thus, while their complexities can both be written $\mathcal{O}(n)$, in practice linear minimizations are much simpler to solve than projections. Figure~\ref{fig:l1} presents a computational comparison. The results are averaged over $5$ runs and the shaded areas represent $\pm1$ standard deviation.

\vspace{2mm}
\begin{figure}[h]
\centering{\includegraphics[scale=0.6]{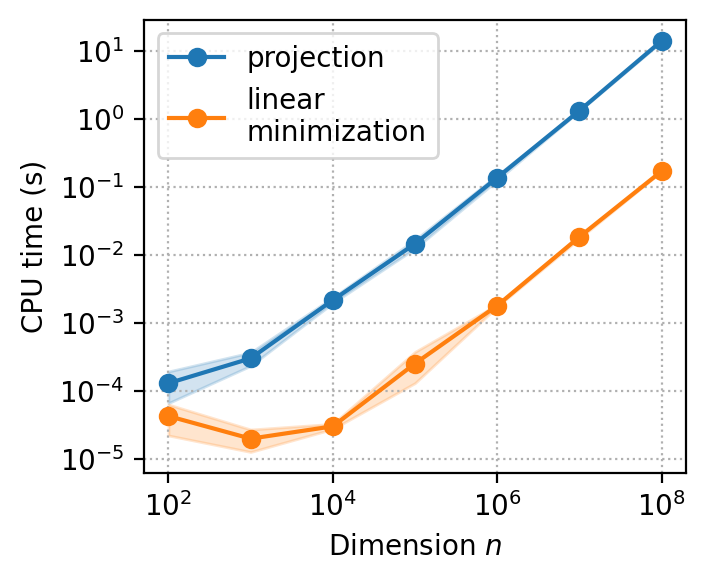}}
\caption{Solving a linear minimization and a projection on the $\ell_1$-ball. The input vector $y\in\mathbb{R}^n$ is generated by sampling entries from the standard normal distribution and the projection method is \cite[Fig.~2]{condat16}, which is state-of-the-art in practice \cite[Tab.~3]{condat16}. In this situation, the plots suggest that linear minimizations are about $100\times$ faster to solve than projections when $n$ is large enough.}
\label{fig:l1}
\end{figure}

\subsection{The \texorpdfstring{$\ell_2$}{l2}-ball and the \texorpdfstring{$\ell_\infty$}{linf}-ball}
\label{sec:l2inf}

For all $x\in\mathbb{R}^n$, $\|x\|_2=\sqrt{\sum_{i=1}^n[x]_i^2}$ and $\|x\|_\infty=\max_{i\in\llbracket1,n\rrbracket}|[x]_i|$.
Here, linear minimizations have no significant advantage over projections as they are all available in closed form. For all $y\in\mathbb{R}^n$,
\begin{align*}
 \argmin_{\|x\|_2\leq1}\,\langle x,y\rangle
 =-\frac{y}{\|y\|_2}
 \quad\text{and}\quad
 \argmin_{\|x\|_2\leq1}\|x-y\|_2
 =\frac{y}{\max\{\|y\|_2,1\}},
\end{align*}
and
\begin{align*}
 \argmin_{\|x\|_\infty\leq1}\,\langle x,y\rangle
 =-\operatorname{sign}(y)
 \quad\text{and}\quad
 \argmin_{\|x\|_\infty\leq1}\|x-y\|_2
 =\operatorname{sign}(y)\min\{|y|,1\}.
\end{align*}

\subsection{The \texorpdfstring{$\ell_p$}{lp}-balls for \texorpdfstring{$p\in\left]1,+\infty\right[$}{p in ]1,+inf[}}
\label{sec:lp}

Let $p\in\left]1,+\infty\right[$. The $\ell_p$-ball is
\begin{align*}
 \{x\in\mathbb{R}^n\mid\|x\|_p\leq1\}
 =\left\{x\in\mathbb{R}^n\;\middle\vert\;\left(\sum_{i=1}^n|[x]_i|^p\right)^{1/p}\leq1\right\}.
\end{align*}
Linear minimizations are available in closed form: by duality, for all $y\in\mathbb{R}^n$,
\begin{align*}
 \argmin_{\|x\|_p\leq1}\,\langle x,y\rangle
 =-\nabla\|\cdot\|_q(y)
 =-\frac{\operatorname{sign}(y)|y|^{q-1}}{\|y\|_q^{q-1}},
\end{align*}
where $q=p/(p-1)\in\left]1,+\infty\right[$.
To the best of our knowledge, there is no projection method specific to the $\ell_p$-ball when $p\in\left]1,2\right[\cup\left]2,+\infty\right[$. We use \cite[Alg.~6.5]{combettes00}, a Haugazeau-like algorithm \cite{haugazeau68} for projecting onto the intersection of sublevel sets of convex functions. For a single sublevel set, the problem reads 
\begin{align}
  \min_{x\in\mathbb{R}^n}\;&\|x-y\|_2^2\label{pb:hgz}\\
 \text{s.t.}\;&g(x)\leq0,\nonumber
\end{align}
and we assume that $g\colon\mathbb{R}^n\to\mathbb{R}$ is convex and differentiable for ease of exposition.
In our case, $g=\|\cdot\|_p^p-1$. Alternatively, one could use a Lagrange multiplier to formulate~\eqref{pb:hgz} as a strongly convex unconstrained problem, but finding the corresponding multiplier may require a considerable effort of tuning; also note that information is usually given in the form $g(x)\leq0$ rather than in the form of a Lagrange multiplier. The method is presented in Algorithm~\ref{hgz}, where for all $a,b\in\mathbb{R}^n$,
\begin{align*}
 H(a,b)=\{x\in\mathbb{R}^n\mid\langle x-b,a-b\rangle\leq0\},
\end{align*}
and $z_t$ is the projection of $x_t$ onto $\{x\in\mathbb{R}^n\mid g(x_t)+\langle x-x_t,\nabla g(x_t)\rangle\leq0\}=H(x_t,z_t)$. If $g(x_t)\leq0$, then $x_{t+s}=x^*$ for all $s\in\mathbb{N}$ \cite[Prop.~3.1]{combettes00}, where $x^*$ is the solution to problem~\eqref{pb:hgz}.

\begin{algorithm}[h]
\caption{Haugazeau-like for problem~\eqref{pb:hgz}}
\label{hgz}
\begin{algorithmic}[1]
\REQUIRE Point to project $y\in\mathbb{R}^n$.
\STATE$x_0\leftarrow y$
\FOR{$t=0$ \textbf{to} $T-1$}
\STATE$z_t\leftarrow x_t-\displaystyle\frac{g(x_t)}{\|\nabla g(x_t)\|_2^2}\nabla g(x_t)\mathds{1}_{\{g(x_t)>0\}}$\label{hgz:G}
\STATE$x_{t+1}\leftarrow\operatorname{proj}(x_0,H(x_0,x_t)\cap H(x_t,z_t))$\label{hgz:proj}
\ENDFOR
\end{algorithmic}
\end{algorithm}

The projection in Line~\ref{hgz:proj} is available in closed form \cite[Thm.~3-1]{haugazeau68}; see also \cite[Cor.~29.25]{plc}. The complexity of an iteration of Algorithm~\ref{hgz} is $\mathcal{O}(n)$. We propose in Theorem~\ref{th:hgz} the convergence rate of Algorithm~\ref{hgz}, based on a key result from \cite[Thm.~7.12]{bauschke96}. A convergence rate is also proposed in \cite{pang15}, however it uses a stronger assumption and has a minor error in the exponent of the constant.

\begin{theorem}
 \label{th:hgz}
 Let $g\colon\mathbb{R}^n\rightarrow\mathbb{R}$ be a differentiable convex function and $\mathcal{C}=\{x\in\mathbb{R}^n\mid g(x)\leq0\}$, and suppose that there exists $\hat{x}\in\mathcal{C}$ such that $g(\hat{x})<0$. Consider Algorithm~\ref{hgz} and let $x^*=\operatorname{proj}(x_0,\mathcal{C})$. Then, for all $t\in\mathbb{N}$,
 \begin{align}
  \|x^*-x_0\|_2^2-\|x_t-x_0\|_2^2
  \leq\frac{\max\{8\rho^2,2\}\|x_0-x^*\|_2^2}{t+2},\label{hgz:rate1}
 \end{align}
 and
 \begin{align}
  \|x_t-x^*\|_2
  \leq\frac{\max\{2\sqrt{2}\rho,\sqrt{2}\}\|x_0-x^*\|_2}{\sqrt{t+2}},\label{hgz:rate2}
 \end{align}
 where $\rho=(-1/g(\hat{x}))\sup_{t\in\mathbb{N}}\|\nabla g(x_t)\|_2\|x_t-\hat{x}\|_2<+\infty$.
\end{theorem}

\begin{proof}
 First, note that by \cite[Prop.~3.1]{combettes00}, for all $t\in\mathbb{N}$,
 \begin{align}
  \|x_t-x_0\|_2
  \leq\|x_{t+1}-x_0\|_2
  \leq\|x^*-x_0\|_2.\label{inc}
 \end{align}
 We prove by induction that~\eqref{hgz:rate1} holds for all $t\in\mathbb{N}$. The base case $t=0$ is trivial. Suppose that~\eqref{hgz:rate1} holds at iteration $t\in\mathbb{N}$. Since $x_{t+1}\in H(x_0,x_t)$ and $x_{t+1}\in H(x_t,z_t)$, we have
 \begin{align}
  \|x_{t+1}-x_0\|_2^2-\|x_t-x_0\|_2^2
  &=\|x_{t+1}-x_t\|_2^2+2\langle x_{t+1}-x_t,x_t-x_0\rangle\nonumber\\
  &\geq\|x_{t+1}-x_t\|_2^2\nonumber\\
  &\geq\operatorname{dist}(x_t,H(x_t,z_t))^2.\label{hgz:1}
 \end{align}
 By \cite[Thm.~7.12]{bauschke96}, 
 \begin{align}
  \operatorname{dist}(x_t,\mathcal{C})
  \leq\rho\operatorname{dist}(x_t,H(x_t,z_t)),\label{hgz:2}
 \end{align}
 where $\rho=(-1/g(\hat{x}))\sup_{t\in\mathbb{N}}\|\nabla g(x_t)\|_2\|x_t-\hat{x}\|_2$, and $\rho<+\infty$ because $(x_t)_{t\in\mathbb{N}}$ converges \cite[Cor.~30.9]{plc}. Now, $x^*=\operatorname{proj}(x_0,\mathcal{C})$ so $\mathcal{C}\subset H(x_0,x^*)$. We can assume that $x_t\neq x^*$, so $x_t\notin H(x_0,x^*)$ by~\eqref{inc}. By \cite[Ex.~29.20]{plc},
 \begin{align*}
  \operatorname{dist}(x_t,H(x_0,x^*))
  =\frac{\langle x_t-x^*,x_0-x^*\rangle}{\|x_0-x^*\|_2}.
 \end{align*}
 Thus,
 \begin{align}
  \operatorname{dist}(x_t,\mathcal{C})
  &\geq\operatorname{dist}(x_t,H(x_0,x^*))\nonumber\\
  &=\frac{\langle x_t-x^*,x_0-x^*\rangle}{\|x_0-x^*\|_2}\nonumber\\
 &=\frac{\langle x_t-x_0,x_0-x^*\rangle+\|x_0-x^*\|_2^2}{\|x_0-x^*\|_2}\nonumber\\
 &\geq\|x_0-x^*\|_2-\|x_t-x_0\|_2\nonumber\\
 &=\|x_0-x^*\|_2-\sqrt{\|x^*-x_0\|_2^2-(\|x^*-x_0\|_2^2-\|x_t-x_0\|_2^2)}\label{hgz:3}\\
 &\geq0,\nonumber
 \end{align}
 where we used the Cauchy-Schwarz inequality in the second inequality and~\eqref{inc} in the last inequality. Let $\epsilon_s=\|x^*-x_0\|_2^2-\|x_s-x_0\|_2^2$ for $s\in\{t,t+1\}$. Combining~\eqref{hgz:1}--\eqref{hgz:3},
 \begin{align*}
  \epsilon_{t+1}
  &\leq\epsilon_t-\frac{1}{\rho^2}\left(\|x_0-x^*\|_2-\sqrt{\|x^*-x_0\|_2^2-\epsilon_t}\right)^2.
 \end{align*}
 Since $\sqrt{\alpha-\beta}\leq\sqrt{\alpha}-\beta/(2\sqrt{\alpha})$ for all $\alpha\geq\beta>0$, we obtain
 \begin{align*}
  \epsilon_{t+1}
  &\leq\epsilon_t-\frac{1}{\rho^2}\left(\|x_0-x^*\|_2-\left(\|x^*-x_0\|_2-\frac{\epsilon_t}{2\|x^*-x_0\|_2}\right)\right)^2\\
  &=\left(1-\frac{\epsilon_t}{4\rho^2\|x^*-x_0\|_2^2}\right)\epsilon_t.
 \end{align*}
 Let $\kappa=\max\{8\rho^2,2\}\|x_0-x^*\|_2^2$. We have $\epsilon_t\leq\kappa/(t+2)$ by the induction hypothesis, so
 \begin{align*}
  \epsilon_{t+1}
  &\leq\left(1-\frac{\epsilon_t}{\kappa/2}\right)\epsilon_t\\
  &\leq
  \begin{cases}
   \displaystyle\frac{\kappa/2}{t+2}&\text{if }\displaystyle\epsilon_t\leq\frac{\kappa/2}{t+2}\\
   \displaystyle\left(1-\frac{1}{t+2}\right)\frac{\kappa}{t+2}&\text{if }\displaystyle\epsilon_t\geq\frac{\kappa/2}{t+2}
  \end{cases}\\
  &\leq\frac{\kappa}{t+3}.
 \end{align*}
 We conclude that~\eqref{hgz:rate1} holds for all $t\in\mathbb{N}$. Then, for all $t\in\mathbb{N}$,
 \begin{align*}
  \|x_t-x^*\|_2^2
  &=\|x^*-x_0\|_2^2-\|x_t-x_0\|_2^2+2\langle x^*-x_t,x_0-x_t\rangle\\
  &\leq\|x^*-x_0\|_2^2-\|x_t-x_0\|_2^2,
 \end{align*}
 because $x^*\in H(x_0,x_t)$, since $\mathcal{C}\subset H(x_0,x_t)$ \cite[Prop.~5.2]{combettes00}. This proves~\eqref{hgz:rate2}.
\end{proof}

In our case, $g=\|\cdot\|_p^p-1$ and $g(0)=-1<0$, so Theorem~\ref{th:hgz} holds with
\begin{align*}
 \rho
 =\sup_{t\in\mathbb{N}}\|\nabla g(x_t)\|_2\|x_t\|_2
 =p\sup_{t\in\mathbb{N}}\|x_t\|_{2(p-1)}^{p-1}\|x_t\|_2<+\infty.
\end{align*}
Therefore, the complexity of an $\epsilon$-approximate projection onto the $\ell_p$-ball is $\mathcal{O}(n\rho^2\|y-x^*\|_2^2/\epsilon^2)$.

\begin{remark}
 If $p\in\left[1,+\infty\right[\cap\mathbb{Q}$, another option, although probably less practical, is to formulate the projection problem as a conic quadratic program and to obtain an $\epsilon$-approximate solution using an interior-point algorithm, with complexity $\mathcal{O}(\operatorname{poly}(n)\ln(1/\epsilon))$ \cite{bental01}.
\end{remark}

\subsection{The nuclear norm-ball}
\label{sec:nuc}

This is probably the most popular example of the computational advantage of linear minimizations over projections in the literature. The nuclear norm, a.k.a.\ trace norm, of a matrix is the sum of its singular values and serves as a convex surrogate for the rank constraint \cite{fazel01trace}. The nuclear norm-ball is the convex hull of rank-$1$ matrices:
\begin{align*}
 \{X\in\mathbb{R}^{m\times n}\mid\|X\|_{\operatorname{nuc}}\leq1\}
 =\operatorname{conv}\{uv^\top\mid u\in\mathbb{R}^m,v\in\mathbb{R}^n,\|u\|_2=\|v\|_2=1\}.
\end{align*}
For all $Y\in\mathbb{R}^{m\times n}$, 
\begin{align*}
 \argmin_{\|X\|_{\operatorname{nuc}}\leq1}\|X-Y\|_{\operatorname{F}}
 =U\operatorname{diag}(\hat{\sigma})V^\top,
\end{align*}
where $Y=U\operatorname{diag}(\sigma)V^\top$ is the singular value decomposition (SVD) of $Y$, $(U,\sigma,V)\in\mathbb{R}^{m\times k}\times\mathbb{R}^k\times\mathbb{R}^{n\times k}$, $k=\min\{m,n\}$, and $\hat{\sigma}$ is the projection of $\sigma$ onto the standard simplex $\Delta_k$. The SVD can be computed with complexity $\mathcal{O}(mn\min\{m,n\}+\min\{m^3,n^3\})$ using the Golub-Reinsch algorithm or the $R$-SVD algorithm \cite[Fig.~8.6.1]{golub13}. On the other hand, a linear minimization requires only a truncated SVD:
\begin{align*}
 \argmin_{\|X\|_{\operatorname{nuc}}\leq1}\,\langle X,Y\rangle
 =\argmax_{\|u\|_2=\|v\|_2=1}\operatorname{tr}((uv^\top)^\top(-Y))
 =\argmax_{\|u\|_2=\|v\|_2=1}u^\top(-Y)v
 =uv^\top,
\end{align*}
where $u$ and $v$ are the top left and right singular vectors of $-Y$. A pair of unit vectors $(u,v)\in\mathbb{R}^m\times\mathbb{R}^n$ satisfying $\sigma_1-u^\top (-Y)v\leq\epsilon$ with high probability can be obtained using the Lanczos algorithm with complexity $\mathcal{O}(\nu\ln(m+n)\sqrt{\sigma_1}/\sqrt{\epsilon})$, where $\sigma_1$ and $\nu$ denote the top singular value and the number of nonzero entries in $-Y$ respectively \cite{jaggi13fw,kuczynski92}. Note that $\nu\leq mn$ and that in many applications of interest, e.g., in recommender systems, $\nu\ll mn$.

In practice, the package ARPACK \cite{arpack} is often used to compute the top pair of singular vectors. Furthermore, if the input matrix $Y$ is symmetric, then the package LOBPCG \cite{lobpcg} can be particularly efficient. Figure~\ref{fig:nuc} illustrates both cases, where linear minimizations are solved to machine precision. The results are averaged over $5$ runs and the shaded areas represent $\pm1$ standard deviation.

\vspace{1mm}
\begin{figure}[h]
\centering{\includegraphics[scale=0.6]{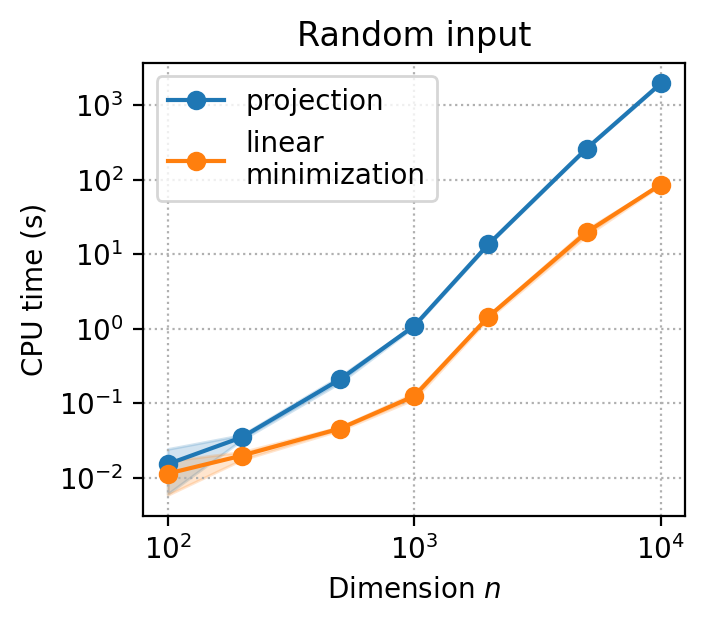}\hspace{15mm}\includegraphics[scale=0.6]{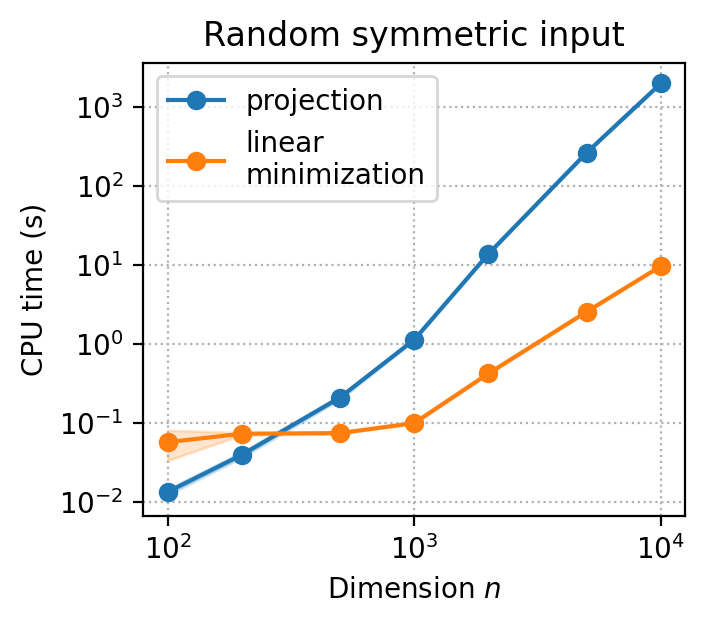}}
\caption{Solving a linear minimization and a projection on the nuclear norm-ball. A matrix $Y\in\mathbb{R}^{n\times n}$ is generated by sampling entries from the standard normal distribution. The full and truncated SVDs are computed using the functions \texttt{svd} and \texttt{svds} from the Python packages \texttt{numpy.linalg} \cite{numpy} and \texttt{scipy.sparse.linalg} \cite{scipy} respectively. The function \texttt{svds} is used with \texttt{tol=0}. \emph{Left:} The input is $Y$ and the function \texttt{svds} is used with \texttt{solver=`arpack'}. \emph{Right:} The input is the symmetric matrix $(Y+Y^\top)/2$ and the function \texttt{svds} is used with \texttt{solver=`lobpcg'}. We see that the ratio of CPU times increases as $n$ increases.}
\label{fig:nuc}
\end{figure}

\subsection{The flow polytope}
\label{sec:flow}

Let $G$ be a single-source single-sink directed acyclic graph (DAG) with $m$ vertices and $n$ edges. Index by $\llbracket1,m\rrbracket$ the set of vertices such that the edges are directed from a smaller to a larger vertex index; this can be achieved with complexity $\mathcal{O}(m+n)$ via topological sort \cite[Sec.~22.4]{cormen09}. Let $A_G\in\mathbb{R}^{m\times n}$ be the incidence matrix of $G$. The flow polytope induced by $G$ is
\begin{align*}
 \mathcal{F}_G=\{x\in\mathbb{R}^n\mid A_Gx=(-1,0,\ldots,0,1)^\top,x\geq0\},
\end{align*}
i.e., $\mathcal{F}_G$ is the set of unit flows $x\in\mathbb{R}^n$ on $G$, where $[x]_i\geq0$ denotes the flow going through edge $i$. Thus, for all $y\in\mathbb{R}^n$, $\argmin_{x\in\mathcal{F}_G}\langle x,y\rangle$ is a flow $x\in\{0,1\}^n$ identifying a shortest path on $G$ weighted by $y$. Its computation has complexity $\mathcal{O}(m+n)$ \cite[Sec.~24.2]{cormen09}. This is significantly cheaper than the complexity $\tilde{\mathcal{O}}(m^3n+n^2)$ of a projection \cite[Thm.~20]{vegh16}, where $\tilde{\mathcal{O}}$ hides polylogarithmic factors.

\subsection{The Birkhoff polytope}
\label{sec:birk}

The Birkhoff polytope, a.k.a.\ assignment polytope, is the set of doubly stochastic matrices 
\begin{align*}
 \mathcal{B}_n=\{X\in\mathbb{R}^{n\times n}\mid X1=1,X^\top1=1,X\geq0\}.
\end{align*}
It is the convex hull of the permutation matrices and arises in matching, ranking, and seriation problems. Linear minimizations can be solved with complexity $\mathcal{O}(n^3)$ using the Hungarian algorithm \cite{kuhn55}. To the best of our knowledge, there is no projection method specific to the Birkhoff polytope so we propose one here. Let $Y\in\mathbb{R}^{n\times n}$. By reshaping it into a vector $y\in \mathbb{R}^{n^2}$, projecting $Y$ onto the Birkhoff polytope is equivalent to solving
\begin{align}
 \begin{aligned}
  \min_{x\in\mathbb{R}^{n^2}}\;&\|x-y\|_2^2\\
 \text{s.t.}\;&Ax=1\\
 &x\geq0,
 \end{aligned}
\quad\quad\text{where}\quad
 A=
 \begin{pmatrix}
  1^\top&0^\top&\cdots&0^\top\\
  0^\top&\ddots&\ddots&\vdots\\
  \vdots&\ddots&\ddots&0^\top\\
  0^\top&\cdots&0^\top&1^\top\\
  I_n&\cdots&\cdots&I_n
 \end{pmatrix}
 \in\mathbb{R}^{2n\times n^2}.\label{a}
\end{align}
This can again be reformulated as
\begin{align}
 \min_{x\in\mathbb{R}^{n^2}}\left(\iota_\mathcal{K}(x)+\frac{1}{2}\|x-y\|_2^2\right)+\left(\iota_\mathcal{A}(x)+\frac{1}{2}\|x-y\|_2^2\right),\label{pb:dr}
\end{align}
where $\mathcal{K}=\{x\in\mathbb{R}^{n^2}\mid x\geq0\}$ and $\mathcal{A}=\{x\in\mathbb{R}^{n^2}\mid Ax=1\}$. That is, we split the constraints into two sets enjoying efficient projections. We can now apply the Douglas-Rachford algorithm \cite{lions79} to problem~\eqref{pb:dr}. The method is presented in Algorithm~\ref{dr:birk}; see Appendix~\ref{apx:birk} for details. Line~\ref{dr:a} computes the projection of $u_t$ onto the affine subspace $\mathcal{A}$ and in Line~\ref{dr:0} is computed a projection onto the nonnegative orthant $\mathcal{K}$. We can set $u=1/n\in\mathcal{A}$ and we denote by $A^\dagger\in\mathbb{R}^{n^2\times2n}$ the Moore-Penrose inverse of $A$.

\begin{algorithm}[h]
\caption{Douglas-Rachford for problem~\eqref{pb:dr}}
\label{dr:birk}
\begin{algorithmic}[1]
\REQUIRE Point to project $y\in\mathbb{R}^{n^2}$, start point $z_0\in\mathbb{R}^{n^2}$, offset point $u\in\mathcal{A}$.
\FOR{$t=0$ \textbf{to} $T-1$}
\STATE$u_t\leftarrow\displaystyle\frac{z_t+y}{2}$\label{dr:u}
\STATE$x_t\leftarrow u_t-A^\dagger A(u_t-u)$\label{dr:a}
\STATE$z_{t+1}\leftarrow\max\left\{\displaystyle\frac{2x_t-z_t+y}{2},0\right\}+z_t-x_t$\label{dr:0}
\ENDFOR
\end{algorithmic}
\end{algorithm}

The complexity of an iteration of Algorithm~\ref{dr:birk} is dominated by the matrix-vector multiplication in Line~\ref{dr:a}. We can assume that $A^\dagger A\in\mathbb{R}^{n^2\times n^2}$ is precomputed. In fact,
\begin{align*}
 A^\dagger A=\frac{1}{n^2}
 \begin{pmatrix}
  B_1&B_2&\cdots&B_2\\
  B_2&\ddots&\ddots&\vdots\\
  \vdots&\ddots&\ddots&B_2\\
  B_2&\cdots&B_2&B_1\\
 \end{pmatrix}
 \in\mathbb{R}^{n^2\times n^2},
 \quad\text{where}\quad
 \begin{cases}
  B_1=nI_n+(n-1)J_n\in\mathbb{R}^{n\times n}\\
  B_2=nI_n-J_n\in\mathbb{R}^{n\times n},
 \end{cases}
\end{align*}
so $A^\dagger A$ is block circulant with circulant blocks (BCCB) and has only three distinct entries: $2n-1$, $n-1$, and $-1$. 
The expression of $A^\dagger A$ can be shown by checking that $A^\dagger=A^\top/n-J_{n^2,2n}/(2n^2)\in\mathbb{R}^{n^2\times2n}$ using the necessary and sufficient Moore-Penrose conditions \cite[Thm.~1]{penrose55}. Thus, the multiplication of $A^\dagger A$ and any vector $x\in\mathbb{R}^{n^2}$ can be performed with complexity $\mathcal{O}(n^2)$. Indeed, it amounts to computing $(nI_n+(n-1)J_n)[x]_{i:j}$ and $(nI_n-J_n)[x]_{i:j}$ for every $(i,j)\in\{(kn+1,(k+1)n)\mid k\in\llbracket0,n-1\rrbracket\}$, each of which has complexity $\mathcal{O}(n)$.

It remains to bound the number of iterations required to achieve $\epsilon$-convergence. Let $x^*=\operatorname{proj}(y,\mathcal{A}\cap\mathcal{K})$ be the solution to problem~\eqref{pb:dr}, i.e., the projection of $Y$ onto the Birkhoff polytope after reshaping, and let $\bar{x}_t=(\sum_{s=0}^tx_s)/(t+1)\in\mathcal{A}$ for all $t\in\mathbb{N}$. By \cite[Thm.~1]{davis17},
\begin{align*}
  \|\bar{x}_t-x^*\|_2
  \leq\frac{\|z_0-z^*\|_2}{\sqrt{2(t+1)}},
\end{align*}
where $z^*$ is a fixed point of $\operatorname{rprox}_{\iota_\mathcal{K}+(1/2)\|\cdot-y\|_2^2}\circ\operatorname{rprox}_{\iota_\mathcal{A}+(1/2)\|\cdot-y\|_2^2}$, $\operatorname{rprox}_\varphi=2\operatorname{prox}_\varphi-\operatorname{id}$, and $\operatorname{prox}_\varphi$ is the proximity operator of $\varphi$ \cite{moreau62}. Therefore, the complexity of an $\epsilon$-approximate projection onto the Birkhoff polytope is $\mathcal{O}(n^2d_z^2/\epsilon^2)$, where
\begin{align}
 d_z=\|z_0-z^*\|_2.\label{dz}
\end{align}

\subsection{The permutahedron}
\label{sec:perm}

Let $\mathfrak{S}_n$ be the set of permutations on $\llbracket1,n\rrbracket$ and $w\in\mathbb{R}^n$. The permutahedron induced by $w$ is the convex hull of all permutations of the entries in $w$, i.e.,
\begin{align*}
 \mathcal{P}_w=\operatorname{conv}\{w_\sigma\in\mathbb{R}^n\mid w_\sigma=([w]_{\sigma_1},\ldots,[w]_{\sigma_{n}})^\top,\sigma\in\mathfrak{S}_n\}.
\end{align*}
It is related to the Birkhoff polytope via $\mathcal{P}_w=\{Xw\mid X\in\mathcal{B}_n\}$.
With no loss of generality, we can assume that the weights are already sorted in ascending order: $[w]_1\leq\cdots\leq[w]_n$. Thus, for all $y\in\mathbb{R}^n$,
\begin{align*}
 \argmin_{x\in\mathcal{P}_w}\,\langle x,y\rangle
 =w_{\sigma^{-1}},
\end{align*}
where $\sigma$ satisfies $[y]_{\sigma_1}\geq\cdots\geq [y]_{\sigma_n}$. Sorting the entries of $y$ has complexity $\mathcal{O}(n\ln(n))$ \cite{cormen09}. A projection can be obtained with a slightly higher complexity $\mathcal{O}(n\ln(n)+n)$, by sorting the entries of $y$ and solving an isotonic regression problem \cite{negrinho14}.

\subsection*{Acknowledgment}

Research reported in this paper was partially supported by the Research Campus MODAL funded by the German Federal Ministry of Education and Research under grant 05M14ZAM.

\appendix

\section{An application of the Douglas-Rachford algorithm}
\label{apx:birk}

Let $\mathcal{H}$ be a Euclidean space with norm $\|\cdot\|$ and denote by $\Gamma_0(\mathcal{H})$ the set of proper lower semicontinuous convex functions $\mathcal{H}\rightarrow\mathbb{R}\cup\{+\infty\}$. 
The Douglas-Rachford algorithm \cite{lions79} can be used to solve
\begin{align*}
 \min_{x\in\mathcal{H}}f(x)+g(x),
\end{align*}
when $f,g\in\Gamma_0(\mathcal{H})$ satisfy $\argmin_\mathcal{H}(f+g)\neq\varnothing$ and $(\operatorname{ri}\operatorname{dom}f)\cap(\operatorname{ri}\operatorname{dom}g)\neq\varnothing$ \cite{plc}, where $\operatorname{ri}$ and $\operatorname{dom}$ denote the relative interior of a set and the domain of a function respectively. It is presented in Algorithm~\ref{dr}. For every function $\varphi\in\Gamma_0(\mathcal{H})$, the proximity operator is $\operatorname{prox}_\varphi=\argmin_{x\in\mathcal{H}}\varphi(x)+(1/2)\|\cdot-x\|^2$ \cite{moreau62}.

\begin{algorithm}[h]
\caption{Douglas-Rachford}
\label{dr}
\begin{algorithmic}[1]
\REQUIRE Start point $z_0\in\mathcal{H}$.
\FOR{$t=0$ \textbf{to} $T-1$}
\STATE$x_t\leftarrow\operatorname{prox}_g(z_t)$\label{prox1}
\STATE$z_{t+1}\leftarrow\operatorname{prox}_f(2x_t-z_t)+z_t-x_t$\label{prox2}
\ENDFOR
\end{algorithmic}
\end{algorithm}

We are interested in an application to problem~\eqref{pb:dr}, where $\mathcal{H}=\mathbb{R}^{n^2}$, $f=\iota_\mathcal{K}+(1/2)\|\cdot-y\|_2^2$, $g=\iota_\mathcal{A}+(1/2)\|\cdot-y\|_2^2$, $y\in\mathcal{H}$, $\mathcal{K}=\{x\in\mathbb{R}^{n^2}\mid x\geq0\}$, $\mathcal{A}=\{x\in\mathbb{R}^{n^2}\mid Ax=1\}$, and $A\in\mathbb{R}^{2n\times n^2}$ is defined in~\eqref{a}. Problem~\eqref{pb:dr} admits a (unique) solution since it is a projection problem onto the intersection of the closed convex sets $\mathcal{K}$ and $\mathcal{A}$, and $1/n\in(\operatorname{ri}\operatorname{dom}f)\cap(\operatorname{ri}\operatorname{dom}g)=(\operatorname{ri}\mathcal{K})\cap\mathcal{A}$ so the Douglas-Rachford algorithm is well defined here. We now show that it reduces to Algorithm~\ref{dr:birk}. For all $t\in\mathbb{N}$,
\begin{align*}
 x_t
 &=\operatorname{prox}_g(z_t)\\
 &=\argmin_{x\in\mathcal{H}}g(x)+\frac{1}{2}\|z_t-x\|_2^2\\
 &=\argmin_{x\in\mathcal{A}}\frac{1}{2}\|x-y\|_2^2+\frac{1}{2}\|x-z_t\|_2^2\\
 &=\argmin_{x\in\mathcal{A}}\left\|x-\frac{z_t+y}{2}\right\|_2^2\\
 &=\operatorname{proj}\left(\frac{z_t+y}{2},\mathcal{A}\right)\\
 &=\frac{z_t+y}{2}-A^\dagger A\left(\frac{z_t+y}{2}-u\right),
\end{align*}
since $\operatorname{proj}(\cdot,\mathcal{A})=\cdot-A^\dagger A(\cdot-u)$, given any $u\in\mathcal{A}$ \cite[Ex.~29.17]{plc}. Thus, Lines~\ref{dr:u}--\ref{dr:a} in Algorithm~\ref{dr:birk} are equivalent to Line~\ref{prox1} in Algorithm~\ref{dr}. Similarly, Line~\ref{dr:0} in Algorithm~\ref{dr:birk} is equivalent to Line~\ref{prox2} in Algorithm~\ref{dr}:
\begin{align*}
 z_{t+1}
 &=\operatorname{prox}_f(2x_t-z_t)+z_t-x_t\\
 &=\argmin_{z\in\mathcal{H}}\left(f(z)+\frac{1}{2}\|2x_t-z_t-z\|_2^2\right)+z_t-x_t\\
 &=\argmin_{z\in\mathcal{K}}\left(\frac{1}{2}\|z-y\|_2^2+\frac{1}{2}\|z-(2x_t-z_t)\|_2^2\right)+z_t-x_t\\
 &=\argmin_{z\in\mathcal{K}}\left\|z-\frac{2x_t-z_t+y}{2}\right\|_2^2+z_t-x_t\\
 &=\operatorname{proj}\left(\frac{2x_t-z_t+y}{2},\mathcal{K}\right)+z_t-x_t\\
 &=\max\left\{\frac{2x_t-z_t+y}{2},0\right\}+z_t-x_t,
\end{align*}
since $\operatorname{proj}(\cdot,\mathcal{K})=\max\{\cdot,0\}$.

\bibliographystyle{abbrv}
{\small\bibliography{biblio}}

\end{document}